\documentclass[11pt]{amsart}
\usepackage{appendix}
\usepackage{amsmath,amsthm,amssymb}
\usepackage{mathrsfs}
\usepackage{color,soul}
\usepackage{cleveref}
\usepackage[all]{xy}
\theoremstyle{plain}
\newtheorem{prop}{Proposition}
\newtheorem{thm}[prop]{Theorem}

\newtheorem{lemma}[prop]{Lemma}
\theoremstyle{remark}
\newtheorem{remark}[prop]{Remark}
\theoremstyle{definition}

\newcommand{\A}{{\mathbb A}}

\newcommand{\Z}{{\mathbb Z}}

\newcommand{\G}{{\mathbb G}}
\newcommand{\PP}{{\mathbb P}}

\newcommand{\cE}{{\mathcal E}}
\newcommand{\cF}{{\mathcal F}}
\newcommand{\cG}{{\mathcal G}}

\newcommand{\cO}{{\mathcal O}}

\newcommand{\eqto}{\stackrel{\lower1.5pt\hbox{$\scriptstyle\sim\,$}}\to}

\DeclareMathOperator{\Spec}{Spec}
\DeclareMathOperator{\Proj}{Proj}

\DeclareMathOperator{\Pic}{Pic}
\DeclareMathOperator{\Br}{Br}
\DeclareMathOperator{\cores}{cores}
\DeclareMathOperator{\image}{im}

\begin{document}
\title[Models of quadric surface bundles]{Models of quadric surface bundles}

\author{Denis Levchenko}
\date{\today}
\address{Institute of Mathematics, University of Zurich, Winterthurerstrasse 190, 8057 Zurich, Switzerland}

\email{denis.levchenko@math.uzh.ch}

\begin{abstract}
We establish the existence of models of quadric surface bundles
with prescribed \'etale local forms.
\end{abstract}

\maketitle

\section{Introduction}
Quadric surface bundles are objects of classical study \cite{staude}.
A special role is played by quadric surface bundles over surfaces
in diverse settings, e.g., \cite{voisin}, \cite{beauville}, \cite{vangeemen},
including questions about (stable) rationality
\cite{ABP}, \cite{HPT}, \cite{schreieder}.

One is interested in the structure of quadric surface bundles,
such as the characterization of smooth models.
Over a surface base, such a result is given in \cite{APS}.

Our main result allows a base of arbitrary dimension.

\begin{thm}
\label{thm.main}
Let $k$ be a perfect field of characteristic different from $2$,
$S$ a smooth projective variety over $k$ of dimension $n$,
and $\pi\colon X\to S$ a morphism of projective varieties which, after base
change $X_{k(S)}\to \Spec(k(S))$ to the generic point of $S$ yields a
smooth quadric surface $X_{k(S)}$.
We assume, over $k$, that resolution of marked ideals on smooth varieties
of dimension $n$ is available.
Then there exists a commutative diagram
\[
\xymatrix{
X' \ar@{-->}[r]^{\varrho_X} \ar[d]_{\pi'} & X \ar[d]^{\pi} \\
S' \ar[r]^{\varrho_S} & S
}
\]
such that
\begin{itemize}
\item $\varrho_S$ is a birational morphism of smooth projective varieties,
\item $\varrho_X$ is a birational map that restricts to an isomorphism over
the generic point of $S$,
\item $\pi'$ is projective and flat and admits a factorization through the
projectivization of a rank $4$ vector bundle on $S'$,
in an \'etale neighborhood
of any point of $S$ taking on one of the normal forms given in Table \ref{table.normalforms}.
\end{itemize}
\end{thm}

\begin{table}
\begin{align*}
K^2-L^2+M^2-N^2&=0,  \\
t_1K^2-L^2+M^2-N^2&=0\ \ \ (n\ge 1),  \\
t_1K^2-t_1L^2+M^2-N^2&=0\ \ \ (n\ge 1),  \\
t_1t_2K^2-t_2L^2+M^2-N^2&=0\ \ \ (n\ge 2),  \\
t_1K^2-L^2+t_2M^2-t_2N^2&=0\ \ \ (n\ge 2),  \\
t_1t_2K^2-t_1L^2+M^2-t_2N^2&=0\ \ \ (n\ge 2),  \\
t_1t_2t_3K^2-t_2L^2+M^2-t_3N^2&=0\ \ \ (n\ge 3),  \\
t_1t_2K^2-t_2t_3L^2+t_3M^2-N^2&=0\ \ \ (n\ge 3). 
\end{align*}
\caption{\'Etale local forms of a
quadric surface bundle in Theorem \ref{thm.main}, over affine space
$\A^n=\Spec(k[t_1,\dots,t_n])$ over the given field $k$,
with projective coordinates $K$, $L$, $M$, $N$}
\label{table.normalforms}
\end{table}

Another point of view, taken in \cite{AB}, takes as its starting point
the embedding of a quadric surface by its anticanonical linear system
as a degree $8$ del Pezzo surface in $\PP^8$, or
\emph{involution surface};
an involution surface is isomorphic to a quadric surface if and only if it
admits a line bundle whose class is half the anticanonical class.
Over a surface base,
involution surface bundles are described in \cite{involution}.

In Section \ref{sec.preliminaries} we recall basic facts about quadric surfaces
and involution surfaces and prove some preliminary results.
In Section \ref{sec.reduction} we give some reduction steps toward the
proof of Theorem \ref{thm.main}.
Section \ref{sec.construction} carries out the construction
of $\widetilde{X}$ and $\widetilde{S}$ in Theorem \ref{thm.main}.
In Section \ref{sec.local} the \'etale local form is justified,
completing the proof of Theorem \ref{thm.main}.
Appendix \ref{app.nonflat} shows that the construction of involution surface
bundles in
\cite{involution} may lead to a non-flat model when attempted
over a base of dimension $>2$.
\subsection*{Acknowledgements}
The author would like to thank his PhD advisor Andrew Kresch for suggesting the study of quadric surface bundles and for the many discussions that followed.
\section{Preliminaries}
\label{sec.preliminaries}
The smooth quadric surface $X_{k(S)}$ uniquely determines a degree
$2$ \'etale $k(S)$-algebra $k(T)$ and an element $\alpha$ of the
Brauer group $\Br(k(T))$, the class of a quaternion algebra that
corestricts to $0$ in $Br(k(S))$ \cite[Example 3.3]{AB}.
Without the condition on corestriction we get a description of a
general involution surface over $k(S)$:
\[ X_{k(S)}\cong R_{k(T)/k(S)}\mathrm{SB}(\alpha), \]
where $\mathrm{SB}(\alpha)$ denotes the Severi-Brauer variety,
the plane conic associated with the
quaternion algebra whose class is $\alpha$,
and $R_{k(T)/k(S)}$ denotes Weil restriction.

The condition on corestriction is an instance of a more general fact
about Weil restrictions of Brauer-Severi varieties.
Geometrically, a Brauer-Severi variety gives rise, upon Weil restriction,
to a variety which is geometrically a product of several copies of
projective space.
On such a variety, a line bundle has a geometric multidegree, a
sequence of integers that is well-defined up to order.

\begin{lemma}
\label{lem.corestriction}
Let $K$ be a field, $L$ a finite \'etale $K$-algebra,
$\alpha\in \Br(L)$, and
$S$ a Brauer-Severi variety over $L$ for $\alpha$.
Set $\beta:=\cores_{L/K}(\alpha)$, and let $d$ be an integer.
Then $d\cores_{L/K}(\alpha)=0$ if and only if
$R_{L/K}(S)$ possesses a line bundle
of geometric multidegree $(d,\dots,d)$.
\end{lemma}

For a Brauer-Severi variety $T\to \Spec(K)$,
the Leray spectral sequence and the fact that $T$, geometrically a
projective space, has geometric Picard group $\Z$ acted upon trivially by
the absolute Galois group of $K$ lead to an exact sequence
\begin{equation}
\label{eqn.lerayT}
0\to \Pic(T)\to \Z\to \ker(\Br(K)\to \Br(T))\to 0
\end{equation}
(see, for instance, \cite[\S 1]{CTO}).
By \cite[V.4.8.3]{giraud},
the element $1\in \Z$ corresponding to the ample generator of the
geometric Picard group of $T$ maps
to minus the class in $\Br(K)$, of which $T$ is a Brauer-Severi variety.

\begin{proof}
To $S$ there is a corresponding central simple algebra over $L$.
A construction in \cite{KO} (see also \cite{ferrand}),
which generalizes the norm of an invertible sheaf \cite[\S 6.5]{EGA2},
supplies the norm algebra, a central simple algebra over $K$
with the Brauer class $\beta$.
We take $T$ to be the associated Brauer-Severi variety.
There is the \emph{twisted Segre embedding}
\[ R_{L/K}(S)\to T, \]
described in \cite{artinbrauerseveri}
in the case $L=K\times\dots\times K$
and straightforward to obtain in the present generality using the functoriality of the
norm construction of \cite{KO}.

Suppose $d\beta=0$.
By \eqref{eqn.lerayT}, there exists a
line bundle on $T$ of geometric degree $d$.
The restriction to $R_{L/K}(S)$ has geometric multidegree $(d,\dots,d)$.

For the converse assertion, we suppose that
$R_{L/K}(S)$ possesses a line bundle of geometric multidegree $(d,\dots,d)$.
Analogous to \eqref{eqn.lerayT} is an exact sequence
\begin{equation}
\label{eqn.lerayR}
0\to \Pic(R_{L/K}(S))\to P\to \ker(\Br(K)\to \Br(R_{L/K}(S)))\to 0;
\end{equation}
where $P\subset \Z^{[L:K]}$ is the group of Galois-invariant multidegrees on
$R_{L/K}(S)$.
By hypothesis, $(d,\dots,d)$ lies in image of the left-hand map of
\eqref{eqn.lerayR}.
We have functorial maps from \eqref{eqn.lerayT} to \eqref{eqn.lerayR},
where $1\in \Z$ is mapped to $(1,\dots,1)\in P$.
This, in turn, is sent by the right-hand map of \eqref{eqn.lerayR} to $-\beta$.
So $d\beta=0$.
\end{proof}

\begin{remark}
\label{rem.arithmeticinvolution}
In \cite[Table 3]{AB} an exhaustive list of
arithmetic possibilities in the case of an involution surface
($[L:K]=2$, $S$ a conic) is given.
\end{remark}

\begin{lemma}
\label{lem.quaterniondescend}
Let $K$ be a field of characteristic different from $2$ with
quadratic extension $L/K$ and class
$\alpha\in \Br(L)$ of a quaternion algebra on $L$.
Then $\cores_{L/K}(\alpha)=0$ if and only if
$\alpha$ is equal to the restriction of the class of a
quaternion algebra on $K$.
\end{lemma}

\begin{proof}
Restriction followed by corestriction is multiplication by $2$.
So,
if $\alpha$ is obtained by restriction from the class of a
quaterion algebra on $K$, then $\cores_{L/K}(\alpha)=0$.
It remains to show, under the assumption that
$\cores_{L/K}(\alpha)=0$, that $\alpha$ is the restriction of the class of a
quaternion algebra on $K$.
We assume $\alpha\ne 0$.

The sequence of cohomology of the Kummer sequence
identifies $H^2(\Spec(K),\mu_2)$ with the
$2$-torsion in the Brauer group $\Br(K)[2]$.
The restriction of scalars $R_{L/K}(\mu_2)$ fits in an exact sequence
\[ 0\to \mu_2\to R_{L/K}(\mu_2)\to \mu_2\to 0. \]
Exactness of the direct image functor under
$\Spec(L)\to \Spec(K)$ \cite[II.3.6]{milne} and the Leray spectral sequence
identify
$H^2(\Spec(K),R_{L/K}(\mu_2))$ with
$H^2(\Spec(L),\mu_2)$, and the long exact sequence of
cohomology identifies the image of the restriction map
\[ \Br(K)[2]\to \Br(L)[2] \]
with the kernel of the corestriction map on $2$-torsion Brauer classes.

Let $\beta\in \Br(K)[2]$, restricting to $\alpha\in \Br(L)$.
Then $\beta$ is a $2$-torsion Brauer class of index $2$ or $4$,
hence is the class of a quaternion division algebra, in which case we are done,
or a biquaterion division algebra, which we suppose now.

Let $d\in K^\times\smallsetminus (K^\times)^2$ be such that
$L\cong K(\sqrt{d})$.
We claim that there exist $a$, $b$, $c\in K^\times$ such that
$\beta=(a,b)+(c,d)$.
The claim implies the lemma, since then
the Brauer class of the quaternion algebra $(a,b)$
restricts to $\alpha$.

According to \cite[Prop.\ 3.2.1]{kahnformesquadratiques},
if an anisotropic quadratic form $q$ over $K$ becomes
isotropic over $K(\sqrt{d})$, then $q$ is similar to
$\langle 1,-d\rangle\perp q'$ for some quadratic form $q'$.
The \emph{Albert form} attached to a pair of quaternion algebras
is anisotropic if and only if the corresponding biquaternion algebra is a
division algebra.
By \cite[Thm.\ 8.1.13]{kahnformesquadratiques},
two Albert forms are similar if and only if they come from the same Brauer class.

We now prove the claim.
Since $\beta$ is the class of a biquaternion division algebra,
restricting to the class of a quaternion algebra in $\Br(L)$,
the Albert form of $\beta$ is anisotropic but becomes isotropic over $L$.
So the Albert form is similar to
$\langle 1,-d,-p,q,r,-dpqr\rangle$ for some $p$, $q$, $r\in K^\times$,
which is similar to the Albert form of
\[ (p,d)\qquad\text{and}\qquad (dpq,dpr). \]
The claim now follows from the result about similarity of Albert forms.
\end{proof}

In Theorem \ref{thm.main},
the existence of a factorization
$\widetilde{X}\to \PP(\cE)\to \widetilde{S}$ of $\tilde\pi$
is claimed, with $\cE$ a rank $4$ vector bundle on $\widetilde{S}$.
Then, on $\PP(\cE)$ we have the line bundles
$\cO_{\PP(\cE)}(n)$, $n\in \Z$, giving rise to line bundles
$\cO_{\widetilde{X}}(n)$ on $\widetilde{X}$ by restriction.
The relative anticanonical sheaf $\omega_{\widetilde{X}/\widetilde{S}}$ is
isomorphic to the tensor product of
$\cO_{\widetilde{X}}(-2)$ and
a line bundle pulled back from $\widetilde{S}$.
So our attention is drawn to a line bundle on $\widetilde{X}$
whose square is isomorphic to
$\omega_{\widetilde{X}/\widetilde{S}}^\vee\otimes \tilde\pi^*\mathcal{M}$ for some
line bundle $\mathcal{M}$ on $\widetilde{S}$.

With attention to the extra data of a line bundle on $\widetilde{X}$
we prove the following extension of \cite[Lemma 2.8]{kreschtschinkelbssurf}.
There,
$A$ is a Henselian ring with residue field $\kappa$, and $X$ and $Y$ schemes, smooth and projective over $\Spec(A)$
with finite group
$G$ of order invertible in $\kappa$, acting on
$X$, $Y$, and (compatibly) on $A$,
with trivial action on $\kappa$.
Under the assumption
\begin{equation}
\label{eqn.H1vanishing}
H^1(X_{\kappa}, T_{X_{\kappa}})=0,
\end{equation}
it is shown that every $G$-equivariant isomorphism
$\phi: X_{\kappa} \to Y_{\kappa}$ arises by restriction from a
$G$-equivariant isomorphism $\psi: X \to Y$ over $\Spec(A)$.
We give a version with line bundles as extra data.

\begin{lemma} \label{lem.ext}
With the above notation and, besides hypothesis \eqref{eqn.H1vanishing}, also
\begin{equation}
\label{eqn.moreH1vanishing}
H^1(X_{\kappa}, \cO_{X_{\kappa}})=0,
\end{equation}
suppose that we further have line bundles $L_X$ on $X$ and $L_Y$ on $Y$,
together with compatible $G$-actions on them and a $G$-equivariant isomorphism
$L_X|_{X_{\kappa}} \cong \phi^*(L_Y|_{Y_{\kappa}})$.
Then there is also a $G$-equivariant isomorphism $L_X\cong \psi^*L_Y$ of line bundles on $X$.
\end{lemma}

\begin{proof}
By \cite[Corollary 18.6.15]{EGA4}, we may assume that $A$ is Noetherian.
To have a $G$-equivariant isomorphism $L_X\cong \psi^*L_Y$ is equivalent to
having a nowhere vanishing $G$-invariant global section of
$M:=L_X^{\vee} \otimes \psi^*L_Y$.
Since a global section may be averaged with its translates under the
group action, it suffices to show that
\[ H^0(X,M)\to H^0(X_{\kappa},M|_{X_{\kappa}}) \]
is surjective.
Letting $\mathfrak{m}$ denote the maximal ideal of $A$, and
$\widehat{A}$ the completion, we consider the commutative diagram
\[
\xymatrix{
H^0(X,M) \ar[r] \ar[d] & H^0(X_\kappa, M|_{X_\kappa}) \\
H^0(X,M) \otimes_A \widehat{A}  \ar[r]^(0.47)\sim & \underset{n}\varprojlim \,H^0(X, M|_{X_n}) \ar[u]
}
\]
where the bottom isomorphism comes from the Theorem on Formal Functions, and
$X_n$ is obtained from $X$ by base change via $\Spec(A/\mathfrak{m}^n)\to \Spec(A)$.
Now it is enough to show $\underset{n}\varprojlim\, H^0(X, M|_{X_n}) \to H^0(X_\kappa, M|_{X_\kappa})$ is surjective.
We claim that \mbox{$H^0(X, M|_{X_{n+1}}) \to H^0(X, M|_{X_n})$}
is surjective for all $n$.
Fix $n$, and let $d$ denote the dimension of $\mathfrak{m}^n/\mathfrak{m}^{n+1}$
as $\kappa$-vector space.
Then we have a short exact sequence of $\cO_X$-modules
\begin{equation*}
0 \to \cO_{X_\kappa}^d \to M|_{X_{n+1}} \to M|_{X_n} \to 0,
\end{equation*}
giving a long exact sequence on cohomology:
\begin{equation*}
0 \to H^0(X_\kappa,\cO_{X_\kappa})^d \to H^0(X, M|_{X_{n+1}}) \to H^0(X, M|_{X_n}) \to H^1(X_\kappa, \cO_{X_\kappa})^d \to \dots
\end{equation*}
From the cohomology vanishing \eqref{eqn.moreH1vanishing} the claim follows.
\end{proof}

\section{Reduction}
\label{sec.reduction}
For a smooth variety $S$, given a Brauer group element $\alpha \in Br(k(S))$, we may consider its \emph{ramification divisor} on $S$, complement of the largest open subscheme of $S$ with a Brauer group element that restricts to $\alpha$. We also use the notion of \emph{root stack} \cite[\S 2]{cadman}, \cite[App. B]{AGV}, specifically the \emph{iterated root stack} along a set of effective Cartier divisors \cite[Def.\ 2.2.4]{cadman}.
\begin{lemma} \label{lem.reduction}
Let $k$ be a perfect field of characteristic different from $2$,
$S$ a smooth projective variety over $k$, $k(T)$ a degree $2$ \'etale $k(S)$-algebra and $T \to S$ the normalization of $S$ in $k(T)$. Let $\alpha \in Br(k(T))$ be a class of a quaternion algebra which corestricts to $0$ in $Br(k(S))$. As in \Cref{thm.main}, we assume that resolution of marked ideals on smooth varieties of dimension equal to $\dim(S)$ is available over $k$. Then there exists a birational morphism $S'\to S$, where $S'$ is smooth projective, such that for $T' \to S'$ the normalization of $S'$ in $k(T)$ we have:
\begin{itemize}
\item the branch divisor $D'=D_1 \cup \dots \cup D_r$ of $T' \to S'$ on $S'$ is smooth.
\item There is a simple normal crossing divisor $D_1 \cup \dots \cup D_r \cup E_1 \cup \dots \cup E_s$ on $S'$, such that for $F_i$'s the pre-images of $E_i$'s on $T'$, $F_1 \cup \dots \cup F_s$ is the ramification divisor of $\alpha$ on $T'$.
\item $\alpha$ is the class of a quaternion algebra over $k(T')$ which is
the restriction of a sheaf of Azumaya algebras of degree $2$ on the iterated root stack $\sqrt{(T',\{F_1,\dots,F_s\})}$.
\item All intersections of triples of divisors $E_1, \dots , E_s$ are empty.
\item The projective representation of
$\mu_2 \times \mu_2$ determined by the sheaf of Azumaya algebras at any component of $F_i \cap F_j$,
$i \neq j$, is faithful.
\end{itemize}
\end{lemma}
Notice that smoothness of the branch divisor on $S'$ implies that $T'$ is smooth.

\begin{proof}
By \Cref{lem.quaterniondescend}, $\alpha$ is the restriction of the class $\alpha_0$ of a quaternion algebra in $\Br(k(S))$. Once a choice of such $\alpha_0$ has been made, we may consider its ramification divisor $E$ on $S$. We also have the branch divisor $D$ of $T \to S$. The proof of \Cref{lem.reduction} is a series of reduction steps on $S$, $D$ and $E$.

As a first reduction step, application of the resolution hypothesis tells us,
\emph{we may suppose that the components of the branch divisor and the
	ramification divisor (for a choice of $\alpha_0$) together form a
	simple normal crossing divisor on $S$}.
We denote by $D_1$, $\dots$, $D_r$ the components of the branch divisor.
The ramification divisor may have components in common with the
branch divisor; we denote its additional components by
$E_1$, $\dots$, $E_s$.
Arguing as in the preparation to the
proof of \cite[Thm.\ 4]{oesinghaus}, we pass to the iterated root stack 
along all the divisor components
\[
\mathcal{R}:=\sqrt{(S,\{D_1,\dots,D_r,E_1,\dots,E_s\})}
\]
and obtain a global (cohomological) Brauer class restricting to $\alpha_0$:
\[ \alpha_0\in
\image\big(\Br(\mathcal{R})\to \Br(k(S))\big). \]

Here is the second reduction step:
\emph{we may suppose that there is a simple normal crossing divisor on $S$
	containing the branch divisor, such that the corresponding iterated root stack
	admits a sheaf of Azumaya algebras whose restriction to the generic point is a
	quaternion algebra with Brauer class $\alpha_0$}.
This uses Raynaud-Gruson flatification \cite{raynaudgruson}, as in
the proof of \cite[Thm.\ 1.2]{kreschtschinkelbssurf}.

The quaternion algebra over $k(S)$ spreads out to a sheaf
of Azumaya algebras on some nonempty open $U\subset S$,
disjoint from the ramification divisor and thus identified with
an open substack of $\mathcal{R}$.
Better, we associate to the (unique by \cite[Prop.\ 2.5(iv)]{antieaumeier})
Brauer class on $\mathcal{R}$ restricting to $\alpha_0$ a gerbe
\[ \mathfrak{G}\to \mathcal{R} \]
banded by $\G_m$, and the sheaf of Azumaya algebras determines by the
machinery of nonabelian cohomology, canonically, a rank $2$ vector bundle on
$\mathfrak{G}\times_{\mathcal{R}}U$ whose projectivization descends to yield
the associated Brauer-Severi scheme over $U$ (see the proof of
\cite[Thm.\ IV.2.5]{milne}).
By \cite[Cor.\ 15.5]{LMB},
this vector bundle is the restriction of a coherent sheaf
$\mathcal{E}$ on $\mathfrak{G}$.
To $\mathcal{E}$ there is a Fitting ideal sheaf, which descends to a
coherent sheaf of ideals on $\mathcal{R}$, corresponding to a closed substack
disjoint from $U$, such that upon blowing up to obtain
$\mathcal{R}'\to \mathcal{R}$
there is a vector bundle on $\mathfrak{G}\times_{\mathcal{R}}\mathcal{R}'$
which restricts to the given vector bundle
on $\mathfrak{G}\times_{\mathcal{R}}U$.

We wish to apply Raynaud-Gruson flatification to $\mathcal{R}'\to S$.
While Raynaud-Gruson flatification is defined for schemes, not stacks,
$\mathcal{R}$ is Zariski locally over $S$
a stack quotient by some power of $\mu_2$,
hence $\mathcal{R}'$ enjoys similar local quotient presentations,
and the machinery of Raynaud-Gruson flatification is local in nature
(see, specifically, \cite[(5.3.3)]{raynaudgruson}), so that we may obtain
a coherent sheaf $\mathcal{I}$ of ideals on $S$, corresponding to a
closed subscheme disjoint from $U$, whose blow-up
$\Proj(\bigoplus \mathcal{I}^j)$
has the flatification property: the stack-theoretic closure of $U$ in
$\mathcal{R}'\times_S\Proj(\bigoplus \mathcal{I}^j)$ is
flat over $\Proj(\bigoplus \mathcal{I}^j)$.
The same holds with $\Proj(\bigoplus \mathcal{I}^j)$ replaced by any
projective variety $\widetilde{S}$ with birational morphism to $S$, restricting to an
isomorphism over $U$ and factoring through $\Proj(\bigoplus \mathcal{I}^j)$.
We may apply resolution to the ideal
$\mathcal{I}$ with marking
$D_1$, $\dots$, $D_r$, $E_1$, $\dots$, $E_s$,
to obtain such $\widetilde{S}$ via a sequence of blow-ups
\[ \widetilde{S}\to \dots \to S, \]
each along a smooth center that has normal crossing with the marking
(to which, after each blow-up, the exceptional divisor gets added).

We consider the branch divisor on $\widetilde{S}$
(of its normalization in $k(T)$), whose components
are the proper transforms of
$D_1$, $\dots$, $D_r$ together with
some subset of the components of the exceptional divisors.
There is also the ramification divisor of
$\alpha_0$ on $\widetilde{S}$.
The components of these, together, form a simple normal crossing divisor
on $\widetilde{S}$.
The normalization morphism of $\mathcal{R}\times_S\widetilde{S}$ factors through
the stack-theoretic closure of $U$ in
$\mathcal{R}'\times_S\widetilde{S}$, because the morphism from the stack-theoretic closure
to $\mathcal{R}\times_S\widetilde{S}$ is representable and finite birational.
The iterated root stack of $\widetilde{S}$ along the components of its
branch and ramification divisors factors through $\mathcal{R}\times_S\widetilde{S}$
hence as well through the normalization,
and we have achieved the second reduction step.

We give a third reduction step: \emph{we may suppose that the
	branch divisor on $S$ is smooth, the
	ramification divisor on $T$ of the Brauer class $\alpha$ is the
	pre-image of a divisor on $S$ which has no common components with
	the branch divisor, the two together forming a simple normal crossing divisor
	on $S$, and $\alpha$ is the class of a quaternion algebra over $k(T)$ which is
	the restriction of a sheaf of Azumaya algebras on the iterated root stack of $T$
	along the components of the ramification divisor}.
Notice that smoothness of the branch divisor on $S$ implies that
$T$ is smooth.

Starting from the outcome of the second reduction step,
with branch divisor $D_1\cup\dots\cup D_r$ and simple normal crossing divisor
$D_1\cup\dots\cup D_r\cup E_1\cup\dots\cup E_s$ on $S$,
it is easy to achieve
smoothness of the branch divisor by repeatedly blowing up a
nonempty locus of the form $D_i\cap D_j$ with $i\ne j$.
By applying Steps 2 and 3 of the proof of \cite[Thm.\ 3.1]{oesinghaus} to
components of intersections of pairs, respectively triples of the
divisors $E_1$, $\dots$, $E_s$, we may suppose that the projective
representation of $\mu_2\times \mu_2$
determined by the sheaf of Azumaya algebras at any component of
$E_i\cap E_j$, $i\ne j$, is faithful, and all intersections of
triples of divisors $E_1$, $\dots$, $E_s$ are empty.
Let $F_i$ denote the pre-image of $E_i$ under $T\to S$.
So, we have a finite \'etale morphism
\[
\sqrt{(T,\{F_1,\dots,F_s\})}\to \sqrt{(S,\{D_1,\dots,D_r,E_1,\dots,E_s\})},
\]
by which the sheaf of Azumaya algebras pulls back to yield, on
$\sqrt{(T,\{F_1,\dots,F_s\})}$, a sheaf of Azumaya algebras
whose restriction to $k(T)$ has Brauer class $\alpha$.
Application of Step 4 of loc.\ cit.\ completes the third reduction step.

Composing the three reduction steps into a single birational morphism $S' \to S$ we achieve the desired properties.
\end{proof}
\section{Construction} \label{sec.construction}
With the starting data of \Cref{thm.main}, we now apply \Cref{lem.reduction} in the case when $k(T)$ is the \emph{discriminant extension} parametrizing rulings of the quadric surface $X_{k(S)}$ and $\alpha \in Br(k(T))$ is the corresponding class of a quaternion algebra that corestricts to zero \cite[Example 3.3]{AB}. 

The sheaf of Azumaya algebras on $\sqrt{(T',\{F_1,\dots,F_s\})}$ of \Cref{lem.reduction} determines a smooth $\PP ^1$-fibration $\Omega\to \sqrt{(T',\{F_1,\dots,F_s\})}$. Together with the finite \'etale morphism $\sqrt{(T',\{F_1,\dots,F_s\})}\to \sqrt{(S',\{D_1,\dots,D_r,E_1,\dots,E_s\})}$, we get, by restriction of scalars, a smooth $(\PP ^1 \times \PP ^1)$-fibration $\Phi\to \sqrt{(S',\{D_1,\dots,D_r,E_1,\dots,E_s\})}$. We see that $\Phi_{k(S)}$ is isomorphic to $X_{k(S)}$, hence we have a birational map $\Phi \dashrightarrow X$ restricting to an isomorphism over the generic point of $S$.  Over the generic point of $S'$, hence as well over a dense open subscheme of $S'$, we have a smooth fibration in quadric surfaces with relatively very ample line bundle of degree 2. This line bundle is the restriction of a line bundle $L$ on the whole smooth $(\PP ^1 \times \PP ^1)$-fibration by \Cref{lem.invertiblesheafextend} below.

\begin{lemma}
\label{lem.invertiblesheafextend}
Let $X$ be a regular Noetherian algebraic stack, and let
$U\subset X$ be an open substack.
Then the restriction map $\Pic(X)\to \Pic(U)$ is surjective.
\end{lemma}

\begin{proof}
There is no loss of generality in supposing $U$ to be dense in $X$.
By \cite[Cor.\ 15.5]{LMB},
any coherent sheaf on $U$ is isomorphic to the restriction
of a coherent sheaf on $X$.
Let an invertible sheaf on $U$ be given,
isomorphic to the restriction of coherent sheaf $\mathcal{F}$ on $X$.
We may replace $\mathcal{F}$ by its double dual without changing the
isomorphism type of the restriction to $U$, so by
\cite[Cor.\ 1.2]{hartshorne} there is no loss of generality in
supposing $\mathcal{F}$ to be reflexive.
Then $\mathcal{F}$ is invertible by \cite[Prop.\ 1.9]{hartshorne}.
\end{proof}

Recall that $D_1 \cup \dots \cup D_r \cup E_1 \cup \dots \cup E_s$ is a simple normal crossing divisor on $S'$, where $D_1 \cup \dots \cup D_r$ is the branch divisor of the normalization $T' \to S'$ on $S'$. We now call $E_1 \cup \dots \cup E_s$ the \emph{ramification divisor} on $S'$ (note that this is not necessarily the ramification divisor of a class of a quaternion algebra in $\Br (k(S'))$ that restricts to $\alpha$ as we saw in Lemma \ref{lem.reduction}). In Lemma \ref{lem.reduction}, we achieved that the branch divisor is smooth and all intersections of triples of components of the ramification divisor are empty. Now, by \cite[Lemma 2.8]{kreschtschinkelbssurf}, the \'etale local isomorphism type of $\Phi$ is uniquely determined
by the action of the stabilizer group, which is $\mu_2$ generically along
the branch and ramification divisors,
$\mu_2\times \mu_2$ generically along an intersection
(of the branch and ramification divisors, or of two components
of the ramification divisor), and is
$\mu_2\times \mu_2\times \mu_2$ at points of the branch divisor
that lie on two components of the ramification divisor.
Twisting $L$ by components of the gerbe over the branch divisor (i.e., the reduced pre-image
of $D_1 \cup \dots \cup D_r$) as appropriate, we may suppose that the space of
invariant sections of $L$ on fibers of $\Phi$ away from the ramification divisor
always has dimension $> 1$.
We may furthermore suppose that the space of
invariant sections of $L$ is nontrivial on fibers of $\Phi$ over the ramification divisor:
this potentially fails on components with $\mu_2 \times \mu_2 \times \mu_2$ stabilizer, which we
blow up, then after blowing up the intersection of the new exceptional divisor with
the proper transform of the gerbe over the branch divisor and twisting $L$ by both new
exceptional divisors there are nontrivial invariant sections of $L$ on all fibers of $\Phi$.

Let $S'^\circ$ be the complement of $S'$ by all the intersections of divisors above, which comprise a locus that, if nonempty, has codimension $2$. We use the same names $D_i$ and $E_j$ (and branch and ramification divisor, respectively) for their intersections with $S'^\circ$ and let $H$ be their union viewed as a divisor on $S'^\circ$. As $H$ is regular, the iterated root stack $\sqrt{(S'^\circ,\{D_1,\dots,D_r,E_1,\dots,E_s\})}$ is just $\sqrt{(S'^\circ, H)}$, for which we have a simple local description by \cite[Section 2.4]{kreschtschinkelbssurf}: $\sqrt{(S'^\circ, H)}$ is, Zariski-locally on $S'^\circ$, the quotient stack of $\Spec(\cO_{S'^\circ}[t]/(t^2-f))$ by the scalar action of $\mu_2$ on $t$, where $f$ is the regular function locally defining $H$. Let $\cG$ be the gerbe of the root stack $\sqrt{(S'^\circ, H)}$, $\Phi^\circ \to \sqrt{(S'^\circ,H)}$ the pullback of the $\PP^1\times \PP^1$-fibration $\Phi$, and let $\cF$ be the fiber of $\Phi^\circ$ over the gerbe. This is summarized in the diagram below, where each square is cartesian apart from the bottom middle one. Then the fibers over points $x \in H$ with residue field $\kappa(x)$ are copies of $B\mu_{2,\kappa(x)}$ in $\cG$ and copies of $[\PP^1\times \PP^1/ \mu_{2, \kappa(x)}]$ in $\cF$, where the action of $\mu_2$ depends on whether $x$ is a point of the branch divisor or the ramification divisor.

\[
\xymatrix{
[\PP^1\times\PP^1/\mu_{2, \kappa(x)}] \ar[r] \ar[d] & \cF \ar[r] \ar[d] & \Phi^\circ \ar[r] \ar[d] & \Phi \ar[d] \\
B\mu_{2,\kappa(x)} \ar[r] \ar[d] & \cG \ar[r] \ar[d] & \sqrt{(S'^\circ, H)} \ar[r] \ar[d] & \sqrt{(S',\{D_1,\dots,D_r,E_1,\dots,E_s\})} \ar[d] \\
\Spec \kappa(x) \ar[r] & H \ar[r] & S'^\circ \ar[r] & S'
}
\]

As in the proof of \cite[Prop.\ 3.1]{kreschtschinkelbssurf}, we will perform a construction by blow-up, contraction and descent, yielding quadric surface bundle $B^\circ$ over $S'^\circ$. We then canonically extend this over the codimension $2$ locus to the whole of $S'$ \cite[\S 9]{kreschtschinkelbssurf}.

Over the points of the branch divisor, $\mu_2$ acts by swapping the factors of $\PP^1\times\PP^1$. We blow up the fixed points under this action, the diagonals $\Delta$ of $\PP^1\times \PP^1$.

Over the points of the ramification divisor, $\mu_2$ acts by $((s_0:s_1),(t_0:t_1)) \mapsto ((s_0:-s_1),(t_0:-t_1))$ for the coordinates on $\PP^1\times\PP^1$. We have four fixed points. The restriction of $L$ to a fiber is $\cO(1,1)$, with $\mu_2$ action given by either multiplication by $-1$ on $s_0\otimes t_1$ and $s_1\otimes t_0$, multiplication by $1$ on $s_0\otimes t_0$ and $s_1 \otimes t_1$, or the opposite, with $1$ and $-1$ swapped. Global sections on the quotient stack are the invariant sections, which form a pencil of conics on $\PP^1\times\PP^1$. The base locus is $\{((0:1),(1:0)), ((1:0),(0:1))\}$ for the first action and $\{((0:1),(0:1)), ((1:0),(1:0))\}$ for the second (together they are the four fixed points of the action on $\PP^1\times\PP^1$). Now we blow up the movable fixed points. 

Let $\widetilde{\Phi}$ be the combined blowup, $E$ the exceptional divisor, and $F$ the proper transform of $\cF$. On $F\subset \widetilde{\Phi}$ the action of $\mu_2$ is nontrivial. So we need to contract $F$; for this we follow \cite[Proposition A.9]{kreschtschinkelbssurf}. 
Let $L_0$ be some appropriate power $a$ of the pullback of $L$ to $\widetilde{\Phi}$ twisted by $F$, so that $L_0$ is ample. We can verify locally over suitable \'etale charts of the root stack that this line bundle gives us the desired contraction. 

If $x$ lies on the branch divisor, we have the fiber $F_x\cong \PP^1\times\PP^1$. The normal bundle $N_{F_x/\widetilde{\Phi}_x}$ is
\begin{equation*}
\cO_{\widetilde{\Phi}_x}(F_x)|_{F_x}\cong \cO_{\widetilde{\Phi}_x}(-E_x)|_{F_x} \cong \cO_{F_x}(-\Delta)\cong \cO_{F_x}(-1,-1).
\end{equation*}
Now $L_0|_{F_x}$ is ample and has bidegree $(a-1,a-1)$ so by \cite[Example A.2 (ii)]{kreschtschinkelbssurf}, fiberwise $F_x$ is contracted to a point.

Over points of the ramification divisor, $F$ is a a degree 6 del Pezzo surface, which contracts to a line as in \cite[Step (i) in Section 2.5]{involution}.

Now we have semiample $L_0((a-1)F)=(L(F))^a$ on $\widetilde{\Phi}$, hence semiample $L(F)$, contracting $F$ to some locus in $B^\circ$, where $B^\circ$ has trivial action of $\mu_2$. The contraction yields, as well, a line bundle $M^\circ$ on $B^\circ$ that pulls back along the contraction to $L(F)$.

As $B^\circ$ has trivial action of $\mu_2$ on it, it descends by \cite[Proposition 2.5]{kreschtschinkelbssurf} to a quadric surface bundle $f^\circ:X'^\circ \to S'^\circ$, together with a line bundle $N^\circ$ on it that pulls back to $M^\circ$.

Finally, we can extend $X'^\circ$ across the codimension $2$ locus to the whole of $S'$ as in the final step of \cite[\S 9]{kreschtschinkelbssurf}. Letting $\iota\colon S'^\circ \to S'$ be the inclusion, we let $N=(\iota \circ f^\circ)_*N^\circ$ and consider the Zariski closure of $X'^\circ$ in $\PP(N^\vee)$, call it $X'$. We get a quadric surface bundle $X'\to S'$. This is the map $\pi'$ of \Cref{thm.main}.

\section{\'Etale local normal forms} \label{sec.local}
We now want to exhibit models for $\Phi \to \sqrt{(S'^\circ,\{D_1,\dots,D_r,E_1,\dots,E_s\})}$, \'etale local over $S'$, on which the construction of \Cref{sec.construction} can be carried out explicitly in coordinates. \'Etale local form of $\Phi$ at points $x\in S'$ depends on $x$.

In case of $x\in S'$, away from
the branch and ramification divisors, the \'etale local form is
given by the first entry in Table \ref{table.normalforms}.

Suppose $x$ lies on the branch divisor, but not on the ramification divisor.
Then we consider $\A^n=\Spec(k[t_1,\dots,t_n])$ with cover,
again $\A^n$, gotten by
adjoining $s_1$ with $s_1^2=t_1$.
Let $Q$ denote the quadric defined by
\[ A^2-B^2+C^2-D^2=0 \]
in $\PP^3=\Proj(k[A,B,C,D])$.
We let $\mu_2$ act on the coordinate $s_1$ of the cover $\A^n$, and
on the coordinate $A$ of $Q$.
The construction, applied to
\[
\xymatrix{
[(\A^n\times Q)/\mu_2] \ar[d] \\
[\A^n/\mu_2] \ar[d] \\
\A^n
}
\]
proceeds by blowing up the locus defined by $s_1=A=0$ and contracting the
proper transform of $Q$ over $s_1=0$ to obtain the map
\[ (s_1,t_2,\dots,t_n,A:B:C:D)\mapsto (t_1,\dots,t_n,A:s_1B:s_1C:s_1D). \]
This descends to give the quadric surface bundle over $\A^n$,
whose defining equation appears as the
second entry in Table \ref{table.normalforms}.

Other cases are handled similarly, with an appropriate number of
$\mu_2$ factors acting on the same number of coordintes $s_i$, where
$s_i^2=t_i$, and a suitable action of the product of $\mu_2$ factors on
the coordinates $A$, $B$, $C$, $D$.
We list the maps that arise out of the construction in each case:
\begin{align*}
(s_1,t_2,\dots,t_n,A:B:C:D)&\mapsto (t_1,\dots,t_n,A:B:s_1C:s_1D), \\
(s_1,s_2,t_3,\dots,t_n,A:B:C:D)&\mapsto
(t_1,\dots,t_n,A:s_1B:s_1s_2C:s_1s_2D), \\
(s_1,s_2,t_3,\dots,t_n,A:B:C:D)&\mapsto
(t_1,\dots,t_n,s_2A:s_1s_2B:s_1C:s_1D), \\
(s_1,s_2,t_3,\dots,t_n,A:B:C:D)&\mapsto
(t_1,\dots,t_n,A:s_2B:s_1s_2C:s_1D), \\
(s_1,s_2,s_3,t_4,\dots,t_n,A:B:C:D)&\mapsto
(t_1,\dots,t_n,A:s_1s_3B:s_1s_2s_3C:s_1s_2D), \\
(s_1,s_2,s_3,t_4,\dots,t_n,A:B:C:D)&\mapsto
(t_1,\dots,t_n,s_3A:s_1B:s_1s_2C:s_1s_2s_3D).
\end{align*}
The respective images correspond to
the remaining entries in Table \ref{table.normalforms}.

\appendix
\section{Non-flatness of involution surface bundles over
varieties of dimension $\ge 3$}
\label{app.nonflat}

Let $k$ be a perfect field of characteristic different from $2$.
Having (see reduction steps) achieved a finite degree $2$ covering
of smooth varieties with associated finite \'etale covering of
root stacks and smooth $\PP^1$-fibration over the source root stack determining
a smooth $\PP^1\times \PP^1$-fibration over the target root stack,
one might ask what kind of fibration arises by
applying the construction from \cite{involution}, that is:
\begin{itemize}
\item away from the ramification divisor we perform the construction of
blow-up, contraction, and descent,
\item away from the branch divisor apply we apply restriction of scalars to an
associated conic bundle,
\item we extend the ambient $\PP^8$-bundle, projectivization of a rank $9$
locally free coherent sheaf, to the whole base by applying the direct image
functor of the inclusion of the complement of the intersection of the
branch and ramification divisors to the locally free coherent sheaf,
\item we extend the involution surface bundle by means of Zariski closure.
\end{itemize}
Here we show that the third step succeeds, in the sense that the direct image
of the locally free coherent sheaf remains locally free;
it is enough to check this \'etale locally and therefore it suffices to
work with \'etale local models as in \cite[\S 3]{involution}.
But we will see that the fourth step leads, generally, to a scheme that
is not flat over the base.

As in loc.\ cit., we begin with \'etale local models over an
arbitrary field $k$ of characteristic different from $2$.
The first local model is taken directly from loc.\ cit.:
\begin{align*}
k[r,r^{-1},s,s^{-1},t]\langle &t^2u^2u'^2,tu^2u'v',t^2u^2v'^2,\\
&tuvu'^2,uvu'v',tuvv'^2, \\
&t^2v^2u'^2,tv^2u'v',t^2v^2v'^2\rangle.
\end{align*}
We write the same, under a change of coordinates, to reflect the situation of
two intersecting components of the ramification divisor:
\begin{align*}
k[r,r^{-1},s,t,t^{-1}]\langle &s^2(u^2+v^2)(u'^2+v'^2),
s(u^2+v^2)(u'^2-v'^2),s^2(u^2+v^2)u'v',\\
&s(u^2-v^2)(u'^2+v'^2),(u^2-v^2)(u'^2-v'^2),s(u^2-v^2)u'v', \\
&s^2uv(u'^2+v'^2),suv(u'^2-v'^2),s^2uvu'v'\rangle.
\end{align*}
With the coordinate $r$ defining the locus where the
degree $2$ covering is ramified we have, from from loc.\ cit.:
\begin{align*}
k[r,s,s^{-1},t,t^{-1}]\langle &r^2u^2u'^2,
r^2(uv'+vu')^2,(uv'-vu')^2,r^2v^2v'^2,\\
&ruu'(uv'-vu'),rvv'(uv'-vu'),r(uv'+vu')(uv'-vu'),\\
&r^2uu'(uv'+vu'),r^2vv'(uv'+vu')\rangle.
\end{align*}
We compute that the intersection of the above modules is the free module
\begin{align*}
k[r,s,t]\langle &r^2s^2uvu'v',
r^2t^2(u^2-v^2)(u'^2-v'^2),
s^2t^2(uv'-vu')^2,\\
&rst(uu'+vv')(uv'-vu'),
r^2st(uu'-vv')(uv'+vu'),\\
&rs^2t(uu'-vv')(uv'-vu'),
rst^2(uv'+vu')(uv'-vu'),\\
&r^2s^2t(uu'+vv')(uv'+vu'),
r^2st^2(uu'-vv')(uu'+vv')\rangle.
\end{align*}

Since we are working with fibrations in surfaces, for the
non-flatness assertion it suffices to show that the
Zariski closure has fiber over $0$ of dimension at least $3$.

For this we may suppose that $k$ is algebraically closed.
Let $\alpha$, $\beta$, $\gamma\in k$ be general, let us consider the curve in
$\A^3\times \PP^1\times \PP^1$ with affine coordinates $r$, $s$, $t$ on $\A^3$
and respective projective coordinates $u:v$ and $u':v'$ on the $\PP^1$ factors
defined by
\begin{align*}
r&=s, \\
t^2&=\frac{1}{16}(-s^4\alpha^2-s^2\beta^2+s^2\gamma^2
+s^4\alpha^2\beta^2\gamma^{-1}), \\
(u:v)&=(s^2\alpha\beta\gamma^{-1}+s^2\alpha+s\beta+s\gamma:4t), \\
(u':v')&=(s^2\alpha\beta\gamma^{-1}-s^2\alpha+s\beta-s\gamma:4t).
\end{align*}
In the image of the curve we find points with projective
coordinates $(x_0:\dots:x_8)$ satisfying
\[ x_3=\alpha x_0,\qquad x_4=\beta x_0,\qquad x_5=\gamma x_0, \]
and, generically, $x_0\ne 0$.
This observation establishes the non-flatness.


\begin{thebibliography}{99}

\bibitem{AGV}
D. Abramovich, T. Graber, and A. Vistoli,
Gromov-Witten theory of Deligne-Mumford stacks,
Amer. J. Math., 130(5):1337–1398, 2008.

\bibitem{antieaumeier}
B. Antieau and L. Meier,
The Brauer group of the moduli stack of elliptic curves,
Algebra Number Theory 14 (2020), 2295--2333.

\bibitem{artinbrauerseveri}
M. Artin,
Brauer-Severi varieties,
in Brauer groups in ring theory and algebraic geometry (Antwerp, 1981),
Lecture Notes in Math. 917, Springer-Verlag, Berlin, 1982, pp. 194--210.

\bibitem{AB}
A. Auel and M. Bernardara,
Semiorthogonal decompositions and birational geometry of
del Pezzo surfaces over arbitrary fields,
Proc. London Math. Soc. 117 (2018), 1--64.

\bibitem{ABP}
A. Auel, Chr. B\"ohning, and A. Pirutka,
Stable rationality of quadric and cubic surface bundle fourfolds,
Eur. J. Math. 4 (2018), 732--760.

\bibitem{APS}
A. Auel, R. Parimala, and V. Suresh,
Quadric surface bundles over surfaces,
Doc. Math. Extra vol.: Alexander S. Merkurjev's sixtieth birthday (2015),
31--70.

\bibitem{beauville}
A. Beauville,
Determinantal hypersurfaces,
Michigan Math. J. 48 (2000), 39--64.

\bibitem{cadman}
C. Cadman,
Using stacks to impose tangency conditions on curves,
Amer. J. Math. 129 (2007), 405--427.

\bibitem{CTO}
J.-L. Colliot-Th\'el\`ene and M. Ojanguren,
Vari\'et\'es unirationnelles non rationnelles:
au-del\`a de l'exemple d'Artin et Mumford,
Invent. Math. 97 (1989), 141--158.

\bibitem{ferrand}
D. Ferrand,
Un foncteur norme,
Bull. S. M. F. 126 (1998), 1--49.

\bibitem{vangeemen}
B. van Geemen,
Some remarks on Brauer groups of $K3$ surfaces,
Adv. Math. 197 (2005), 222--247.

\bibitem{giraud}
J. Giraud,
\emph{Cohomologie non ab\'elienne},
Springer-Verlag, Berlin, 1971.

\bibitem{EGA2}
A. Grothendieck,
\'El\'ements de g\'eom\'etrie alg\'ebrique, II:
\'Etude globale \'el\'ementaire de quelques classes de morphismes,
Publ. Math. IHES 8 (1961).

\bibitem{EGA4}
A. Grothendieck,
El\'ements de g\'eom\'etrie alg\'ebrique IV,
Publ. Math. IHES 20, 24, 28, 32, (1964--67).



\bibitem{hartshorne}
R. Hartshorne,
Stable reflexive sheaves,
Math. Ann. 254 (1980), 121--176.

\bibitem{HPT}
B. Hassett, A. Pirutka, and Yu. Tschinkel,
Stable rationality of quadric surface bundles over surfaces,
Acta Math. 220 (2018), 341--365.

\bibitem{kahnformesquadratiques}
B. Kahn,
\emph{Formes quadratiques sur unn corps},
Soc. Math. France, Paris, 2008.

\bibitem{KO}
M.-A. Knus and M. Ojanguren,
A norm for modules and algebras,
Math. Z. 142 (1975), 33--45.

\bibitem{involution} A. Kresch and Yu. Tschinkel,
Involution surface bundles over surfaces, Math. Z. 296 (2020), 1081--1100.

\bibitem{kreschtschinkelbssurf}
A. Kresch and Yu. Tschinkel,
Models of Brauer-Severi surface bundles,
Moscow Math. J. 19 (2019), 549--595.

\bibitem{LMB}
G. Laumon and L. Moret-Bailly,
\emph{Champs Alg\'ebriques},
Springer-Verlag, Berlin, 2000.

\bibitem{milne}
J. S. Milne,
\emph{\'Etale Cohomology},
Princeton Univ. Press, Princeton, 1980.

\bibitem{oesinghaus}
J. Oesinghaus,
Conic bundles and iterated root stacks,
Eur. J. Math. 5 (2019), 518--527.

\bibitem{raynaudgruson}
M. Raynaud and L. Gruson,
Crit\`eres de platitude et de projectivit\'e,
Invent. Math. 13 (1971), 1--89.

\bibitem{schreieder}
S. Schreieder,
Quadric surface bundles over surfaces and stable rationality,
Algebra Number Theory 12 (2018), 479--490.

\bibitem{staude}
O. Staude,
Flächen 2. Ordnung und ihre Systeme und Durchdringungskurven,
in Enzykl. der Mathem. Wissenschaften III $2_{\mathrm{I}}$,
B. G. Teubner Verlag, Leipzig, 1903--1915, pp. 161--256.



\bibitem{voisin}
C. Voisin,
Th\'eor\`eme de Torelli pour les cubiques de ${\mathbf P}^5$,
Invent. Math. 86 (1986), 577--601.

\end{thebibliography}
\end{document}